\documentclass[12pt,a4paper]{amsart}

\usepackage{amscd,amssymb,amsthm}
\usepackage{graphicx}
\theoremstyle{plain}
\newtheorem{theorem}{Theorem}[section]
\newtheorem{corollary}[theorem]{Corollary}
\newtheorem{lemma}[theorem]{Lemma}
\newtheorem{proposition}[theorem]{Proposition}

\theoremstyle{definition}

\theoremstyle{remark}

\numberwithin{equation}{section}

\numberwithin{table}{section}

\numberwithin{figure}{section}

\usepackage{enumerate}
\newcommand{\oo}[1]{\buildrel {\scriptscriptstyle\circ}\over{#1}}
\newcommand{\reel}{\mathbb{R}}
\newcommand{\ds}{\displaystyle}
\newcommand{\abs}[1]{\left\vert #1\right\vert }

\title[Ratios of Differences of Power Means ]
{Bounds for the Ratios of Differences of Power Means 
  in Two Arguments}
\author{Omran Kouba}
\address{Department of Mathematics \\
Higher Institute for Applied Sciences and Technology\\
P.O. Box 31983, Damascus, Syria.}
\email{omran\_kouba@hiast.edu.sy}


\begin{document}

\begin{abstract}
Using methods from classical analysis, sharp bounds for the ratio of differences of Power Means are obtained. Our results 
generalize and extend previous ones due to S. Wu(2005), and
to S. Wu and L. Debnath.
\end{abstract}

\maketitle

\section{Introduction }\label{sec1}

Given two distinct positive real numbers $a$ and $b$, we recall that the arithmetic mean $A(a,b)$, the geometric mean
$G(a,b)$, the identric mean $I(a,b)$,  and finally  $M_r(a,b)$,  the power mean of order $r$,  are respectively defined by
$$
A(a,b)=\frac{a+b}{2},~~ G(a,b)=\sqrt{ab},~~ I(a,b)=\frac{1}{e}\left(\frac{a^a}{b^b}\right)^{1/(a-b)}
~\hbox{and}~ M_r(a,b)=\root{r}\of{\frac{a^r+b^r}{2}}.
$$

Inequalities relating means in two arguments attracted and continue to attract the attention of mathematicians.
Many recent papers  were concerned in comparing the differeces between known means. 

For instance, H. Alzer and S. Qui
proved in \cite{alqu} the following inequality relating the identric, geometric and arithmetic means of two distinct positive numbers $a$ and $b$ :
$$\frac{2}{3} <\frac{I(a,b)-G(a,b)}{A(a,b)-G(a,b)}<\frac{2}{e}.$$

This was later complemented by T. Trif \cite{trif} who proved that, for $p\geq2$ and every
distinct positive numbers $a$ and $b$, we have
$$\left(\frac{2}{e}\right)^p<\frac{I^p(a,b)-G^p(a,b)}{A^p(a,b)-G^p(a,b)}<\frac{2}{3}$$

In another direction we proved in \cite{kou} that
the inequality
$$\frac{I^p(a,b)-G^p(a,b)}{A^p(a,b)-G^p(a,b)}<\frac{2}{3}$$
holds true for every distinct positive numbers $a$ and $b$, if and only if $p\geq\ln\left(\frac{3}{2}\right)/\ln\left(\frac{e}{2}\right)\approx 1.3214$, and that the reverse inequality holds true for every distinct positive numbers $a$ and $b$, if and only if $p\leq6/5= 1.2$.

In the same line of ideas, S. Wu proved in \cite{wu1} a double inequality
$$2^{1-1/r}<\frac{M_r(a,b)-G(a,b)}{A(a,b)-G(a,b)}<\left(\frac{2r}{1-r}\right)^{1-1/r},$$
where $0<r<1/2$, $a$ and $b$ are distinct positive real numbers.

Later,  this inequality  was sharpened and extended by Wu and Debnath \cite{wudeb} who proved that for
 distinct positive real numbers $a$ and $b$ we have
$$2^{1-1/r}<\frac{M_r(a,b)-G(a,b)}{A(a,b)-G(a,b)}<r,$$
if $0<r<1/2$ or $r>1$, and a reverse inequality holds if $1>r>1/2$.

In this work, (see  Theorem~\ref{th31}), we will we extend and generalize this further, by finding sharp bounds,
 depending only on $s$, $t$ and $p$,  for the ratio
$$\frac{M_s^p(a,b)-G^p(a,b)}{M_t^p(a,b)-G^p(a,b)}$$
where $s$, $t$ and $p$ are real parameters satisfying some conditions. 
 
S. Wu and L. Debnath in \cite{wudeb} proved also the following inequality
$$\frac{2^{-p/r}-2^{-p/s}}{2^{-p/t}-2^{-p/s}}<\frac{M_r^p(a,b)-M_s^p(a,b)}{M_t^p(a,b)-M_s^p(a,b)}<\frac{r-s}{t-s}.$$
for every distinct positive real numbers $a$ and $b$, provided that $r>t>s>0$, $t\geq p>0$.
 In our Theorem~\ref{th33}, we will extend the domain of validity of this inequality. 

The paper is organized as follows. In section~\ref{sec2}, we gathered some preliminary lemmas and propositions, 
some of them are of interest in their own right. In section~\ref{sec3}, we find the statements and proofs of the main theorems.

\section{ Preliminaries }\label{sec2}

 Our main tool in this investigation is the following Lemma~\ref{lm21}, called the L'Hospital Monotone Rule. It was discovered,
rediscovered, and refined by several mathematicians. a good account of this can be found in \cite{and} where it is traced back
to the work of  Cheeger et al. \cite{grom}. We will include for the convenience of the reader, a proof which is -to our knowledge-different from the known ones.

\bigskip

\noindent {\bf Remark.}
 When we describe a function by saying that it is increasing, decreasing, or convex, we mean that it has this property in the {\it strict} sense. 

\begin{lemma}\label{lm21}
 {\em (L'Hospital Monotone Rule.)}~ Let $I$ be an interval in $\reel$, and let $\oo{I}$ be its interior. We consider two continuous functions $f$ and $g$ defined on $I$ and differentiable on $\oo I$, such that for every $x\in\,\oo{I}$ we have $g^\prime(x)\ne 0$. If  $ f^\prime/g^\prime$ is  increasing ({\it resp.}  decreasing) on $\oo{I}$, then for every
$c\in I$ the function :\par
\centerline{$\ds x\mapsto \frac{f(x)-f(c)}{g(x)-g(c)}$}\par
\noindent is increasing ({\it resp.}  decreasing) on $I\setminus\{c\}$.
\end{lemma}

\begin{proof} 
Since a derivative has the Darboux property, we know that $g^\prime(\oo I)$ is an interval, which does not contain $0$ by assumption. So, $g^\prime$ has a constant sign on $\oo I$. Replacing $(f,g)$ by $(-f,-g)$ if necessary, we may assume that $\forall x\in\oo I,~ g^\prime(x)>0$. Similarly, replacing $(f,g)$ by $(-f,g)$ if necessary, we may assume also that  $f^\prime/g^\prime$ is  increasing on $\oo I$. Hence, without loss of generality, it is sufficient to prove the case where  $g^\prime$ is positive and $ f^\prime/g^\prime$ is  increasing.

 Let $J=g(I)$. Since $g$ is continuous and  increasing, it defines a homeomorphism $g:I\longrightarrow J$. Moreover,
since $\forall\, x\in \oo I, g^\prime (x)>0$ we conclude that $g^{-1}$ has a derivative on $\oo J$ with $(g^{-1})^\prime=1/g^\prime\circ g^{-1}$.

Noting that the  continuous function $f\circ g^{-1}$ is derivable on $\oo J$ with
$$\forall\,x\in \oo J,\qquad(f\circ g^{-1})^\prime(x)=\frac{f^\prime(g^{-1}(x))}{g^\prime(g^{-1}(x))}
=\left(\frac{f^\prime}{g^\prime}\right)\circ g^{-1}(x)$$ 
we conclude that $(f\circ g^{-1})^\prime$ is  increasing on $\oo J$. This proves that $f\circ g^{-1}$ is convex on $J$, and this is equivalent to the fact that, for every $\gamma$ in $J$, the function 
$$t\mapsto \frac{f(g^{-1}(t))-f(g^{-1}(\gamma))}{t-\gamma}$$
is  increasing on $J\setminus\{\gamma\}$. Making the increasing change of variable $t=g(x)$ and $\gamma=g(c)$ we conclude that, for
every $c$ in $I$, the function $$x\mapsto \frac{f(x)-f(c)}{g(x)-g(c)}$$ is  increasing on $I\setminus\{c\}$, which is
the desired conclusion.
\end{proof}

In the next Lemma~\ref{lm22} we will introduce and study some properties of a family of functions.

\begin{lemma}\label{lm22}
Given three real numbers $\alpha,\beta,\gamma$ such that $\abs{\alpha}$, $\abs{\beta}$ and $\abs{\gamma}$ are distinct, we consider the function $L_{\alpha,\beta,\gamma}$ defined on $J=(0,+\infty)$ by
$$L_{\alpha,\beta,\gamma}(x)=\frac{\cosh(\alpha x)-\cosh(\gamma x)}{\cosh(\beta x)-\cosh(\gamma x)}.$$
The functions $L_{\alpha,\beta,\gamma}$ satisfy the following properties :
\begin{enumerate}[\upshape(a)]
\item  $L_{\alpha,\beta,\gamma}=L_{\abs{\alpha},\abs{\beta},\abs{\gamma}}$.
\item  $L_{\alpha,\beta,\gamma}(J)$ is contained in exactly one of the intervals $(-\infty,0)$, $(0,1)$, or $(1,+\infty)$.
\item   $\ds L_{\alpha,\beta,\gamma}=\frac{1}{L_{\beta,\alpha,\gamma}}$,
$L_{\alpha,\beta,\gamma}=1-L_{\alpha,\gamma,\beta}$, and
$L_{\alpha,\beta,\gamma}=\ds\frac{L_{\gamma,\beta,\alpha}}{L_{\gamma,\beta,\alpha}-1}$.
\item   If  $0\leq \gamma<\beta<\alpha$ then $L_{\alpha,\beta,\gamma}$ is  increasing on $J$.
\end{enumerate}
\end{lemma}

\begin{proof}
(a) is obvious. (b) follows from the continuity of $L_{\alpha,\beta,\gamma}$ and the fact that
$L_{\alpha,\beta,\gamma}(x)\notin\{0,1\}$ for every $x>0$.  (c) is a simple verification. Finally, to see (d) 
we note that, if $a>b>0$ then, from the fact that
$$\frac{\cosh(a x)}{\cosh (b x)}=\cosh( (a-b) x)+\sinh( (a-b) x) \tanh(bx)$$
we conclude that $x\mapsto \frac{a\cosh(a x)}{b\cosh(b x)}$ is  increasing on $(0,+\infty)$, and using Lemma~\ref{lm21}. we find that
$x\mapsto \frac{\sinh(a x)}{\sinh(b x)}$ is  increasing on $(0,+\infty)$. 

Now, noting that
$$L_{\alpha,\beta,\gamma}(x)=\frac{\sinh\left(\frac{\alpha+\gamma}{2}x\right)}{\sinh\left(\frac{\beta+\gamma}{2}x\right)}\times
\frac{\sinh\left(\frac{\alpha-\gamma}{2}x\right)}{\sinh\left(\frac{\beta-\gamma}{2}x\right)}$$
we conclude that, if $0\leq\gamma<\beta<\alpha$ then   $L_{\alpha,\beta,\gamma}$ is  increasing on $(0,+\infty)$, as the product of two positive  increasing functions. This ends the proof of Lemma~\ref{lm22}.
\end{proof}

 We are interested in the properties of monotony of these functions. The following result is a simple consequence
of Lemma~\ref{lm22}. 

\begin{corollary}\label{cor23}
For distinct $\abs{\alpha}$, $\abs{\beta}$ and $\abs{\gamma}$, the function
 $\Delta(\alpha,\beta,\gamma)L_{\alpha,\beta,\gamma}$ where
$$\Delta(\alpha,\beta,\gamma)=\hbox{\rm sgn}\,\big((\alpha^2-\beta^2)(\alpha^2-\gamma^2)(\beta^2-\gamma^2)\big)$$
is  increasing on $(0,+\infty)$.
\end{corollary}

\begin{proof}
 By Lemma~\ref{lm22},  if $0\leq\abs{\gamma}<\abs{\beta}<\abs{\alpha}$ then $\Delta(\alpha,\beta,\gamma)=+1$ and $L_{\alpha,\beta,\gamma}$ is  increasing so the conclusion is true in this case.

\smallskip
 Now, if $\Delta(\alpha,\beta,\gamma)L_{\alpha,\beta,\gamma}$ is  increasing, and 
 two of the numbers $\alpha,\beta,\gamma$ are transposed  the quantity $\Delta(\alpha,\beta,\gamma)$ changes its sign, and by Lemma~\ref{lm22}, the function $L_{\alpha,\beta,\gamma}$ is composed with a decreasing function (namely $x\mapsto 1/x$, $x\mapsto1-x$ or $x\mapsto x/(x-1)$). Hence, the product remains  increasing. From
 this, the result follows since these transpositions generate all permutations of $(\alpha,\beta,\gamma)$.
\end{proof}

The following Proposition and its Corollary are the main technical results used in the proof of Theorem~\ref{th31}.

\begin{proposition}\label{pro24}
 Let $r$ and $q$ be two real numbers such that $r\notin\{0,1\}$ and $q\ne 0$, and 
let $G_{r,q}$ be the function defined on $J=(0,+\infty)$ by 
$$G_{r,q}(x)=\frac{\big(\cosh(r x)\big)^{q/r}\tanh(rx)}{\big(\cosh( x)\big)^{q}\tanh(x)}.$$ 
Then, the following conclusion holds :
\begin{enumerate}[\upshape(a)]
\item $G_{r,q}$ is decreasing if and only if $(r,q)$ belongs to one of the following sets :

$
\begin{matrix}
{\scriptscriptstyle \square}~&  \left\{(u,v):u<0,v\leq\min(0,\frac23(u+1))\right\}\setminus\{(-1,0)\},\hfill\\
{\scriptscriptstyle \square}~& \left\{(u,v):0<u<1,v\geq\max(2u,\frac23(u+1))\right\}\setminus\{(\frac12,1)\},\hfill\\
{\scriptscriptstyle \square}~& \left\{(u,v):1<u,v\leq\min(2,\frac23(u+1))\right\}\setminus\{(2,2)\}.\hfill
\end{matrix}$
\item  $G_{r,q}$ is increasing if and only if $(r,q)$ belongs to one of the following sets :

$
\begin{matrix}{\scriptscriptstyle \square}~& \left\{(u,v):u<0,v\geq \max(0,\frac23(u+1))\right\}
\setminus\{(-1,0)\},\hfill\\
{\scriptscriptstyle \square}~& \left\{(u,v):0<u<1, v\leq\min(2u,\frac23(u+1))\right\}\setminus\{(\frac12,1)\},\hfill\\
{\scriptscriptstyle \square}~&\left\{(u,v):1<u,v\geq\max(2,\frac23(u+1))\right\}\setminus\{(2,2)\}.\hfill
\end{matrix}
$
\end{enumerate}
\end{proposition}
\begin{figure}[!h]
\begin{center}
\includegraphics[width=0.55\textwidth]{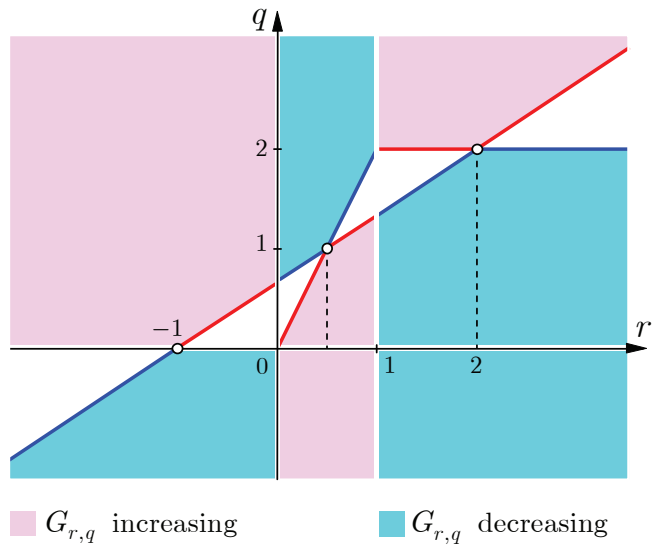}
\caption{The domains where $G_{r,q}$ is strictly monotonous on $(0,+\infty)$ .}\label{fig1}
\end{center}
\end{figure}
\begin{proof} 
First, note that
\begin{align*}
G_{r,q}^\prime(x)&=\frac{(\cosh(rx))^{-2+q/r}}{(\cosh x)^{q}\sinh^2 x}\big(
q\sinh^2(rx)\sinh x\cosh x-q\cosh (rx)\sinh(rx)\sinh^2x\\
&\qquad\qquad+ r\sinh x\cosh x-\cosh(rx)\sinh(rx)\big),\\
&=\frac{(\cosh(rx))^{-2+q/r}}{4(\cosh x)^{q}\sinh^2 x}\big(q(\cosh(2rx)-1)\sinh(2 x)
-q\sinh(2rx)(\cosh (2x)-1)\\
&\qquad\qquad +2r\sinh(2x)-2\sinh(2rx)\big),\\
&=\frac{(\cosh(rx))^{-2+q/r}}{4(\cosh x)^{q}\sinh^2 x}
\big((q-2)(\sinh(2rx)-r\sinh(2 x)) \\
&\qquad\qquad\qquad\qquad-q(\sinh(2(r-1) x)-(r-1)\sinh(2 x))\big).
\end{align*}

So, if we define  $K_\ell$ by $K_\ell(x)=\sinh(2\ell x)-\ell \sinh(2x)$, we find that
$$
G_{r,q}^\prime(x)=\frac{(\cosh(rx))^{-2+q/r}}{4(\cosh x)^{q}\sinh^2 x}H_{r,q}(x),
$$
where $H_{r,q}(x)=(q-2)K_r(x) -qK_{r-1}(x)$. So, 
the sign of $G_{r,q}^\prime(x)$ is the same as that of $H_{r,q}(x)$.

We want to find the necessary and sufficient conditions, on $r$ and $q$,  for $H_{r,q}$ to keep a constant sign on  $J$.

First, let us eliminate some simple cases :
\begin{itemize}
\item For $x\in J$, we have $H_{-1,q}(x)=qK_{2}(x)=2q\sinh(2x)(\cosh(2x)-1)>0$, and
$G_{-1,q}$ is increasing on $J$.
\item For $x\in J$, we have 
$H_{1/2,q}(x)=2(q-1)K_{1/2}(x) =2(1-q)\sinh( x)(\cosh(x)-1)$. So, $G_{1/2,q}$ is increasing on $J$ if $q<1$ and 
it is decreasing on $J$ if $q>1$.
\item  For $x\in J$, we have 
$H_{2,q}(x)=(q-2)K_2(x) =2(q-2)\sinh(2x)(\cosh(2x)-1)$. So, $G_{2,q}$ is increasing on $J$ if $q>2$ and 
it is decreasing on $J$ if $q<2$.
\end{itemize}
In what follows we will suppose that $r\notin\{-1,0,\frac12,1,2\}$.

Since $K_r^\prime(x)=2r(\cosh(2rx)-\cosh(2x))$ has the
sign of $r(r^2-1)$ for every $x>0$, and $K_r(0)=0$,  we conclude that $K_r$
has the sign of $r(r^2-1)$ on $J$, In particular, it does not vanish on this interval. So, let us write
$$H_{r,q}=K_r\times\left(q\left(1-\frac{K_{r-1}}{K_r}\right)-2\right).$$

Clearly, the range of the function $\frac{K_{r-1}}{K_r}$ plays an important role in our study. Noting that $\frac{K_{r-1}^\prime(x)}{K_r^\prime(x)}=\frac{r-1}{r}L_{r-1,r,1}(2x)$,  (where $L_{\alpha,\beta,\gamma}$ is the function defined in Lemma~\ref{lm22}), and remembering that 
$$\Delta(r-1,r,1) =-\hbox{sgn}\big((2r-1)r(r-2)(r^2-1)\big)$$
we conclude, using Corollary~\ref{cor23}, that 
$(2r-1)(r-2)(r+1)\frac{K_{r-1}^\prime}{K_r^\prime}$ is decreasing. Now, Lemma~\ref{lm21}, proves that
$(2r-1)(r-2)(r+1) \frac{K_{r-1}}{K_r}$ is also decreasing,  or, equivalently,
 that $(2r-1)(r-2)(r+1) \big(1-\frac{K_{r-1}}{K_r})$ is increasing. So, let $\widetilde{K}_r=1-K_{r-1}/K_r$, to find its range, we only need to determine the limits of this function at $0^+$ and at $+\infty$. The results are shown in Table~\ref{tab1}.

\begin{table}[h]
\begin{center}
\renewcommand{\arraystretch}{1.25}
\begin{tabular}{|c|c|c|c|}\hline
$r$& $\ds\lim_{0^+}\widetilde{K}_r$&$\ds\widetilde{K}_r$&
$\ds\lim_{\infty}\widetilde{K}_r$\\
\hline\hline
$2<r$             & $\frac{3}{r+1}$ & $\ds\nearrow$ &1  \\
\hline
$1<r<2$         &$\frac{3}{r+1}$  & $\ds\searrow$ &1  \\
\hline
$\frac12<r<1$ &$\frac{3}{r+1}$  & $\ds\searrow$&$\frac{1}{r}$\\ \hline
$0<r<\frac12$  &$\frac{3}{r+1}$ & $\ds\nearrow$ &$\frac{1}{r}$\\ \hline
$-1<r<0$         &$\frac{3}{r+1}$ & $\ds\nearrow$ &$+\infty$ \\
\hline
$r<-1$             &$\frac{3}{r+1}$ & $\ds\searrow$ &$-\infty$ \\
\hline
\end{tabular}
\belowcaptionskip=0pt
\caption{The monotony of $\widetilde{K}_r$, and its limits at $0^+$ and $+\infty$ according to the values of $r$.}
\label{tab1}
\end{center}
\end{table}

In particular, $\widetilde{K}_r$ does not vinish on $J$, it is positive if $r>-1$, and it is negative if $r<-1$. Now, If
the functions $K_r\widetilde{K}_r$ and $2/\widetilde{K}_r$ are denoted by $A_r$ and $B_r$ respectively, then,
 we have $H_{r,q}=A_r\times (q-B_r)$, and using the information in Table~\ref{tab1},
we can determine exactly, under what conditions $H_{r,q}$ keeps a constant sign on $J$. This is summarized in Table~\ref{tab2}, where 
we put, according to the value of $r$, the necessary and sufficient conditions that ensure the validity
of the inequalities  $H_{r,q}>0$ or $H_{r,q}>0$ on $J$. 

\begin{table}[h]
\begin{center}
\renewcommand{\arraystretch}{1.25}
\begin{tabular}{|c|c|c|l|l|}\hline
$r$        &    $\hbox{sgn}(A_r)$  &  $\hbox{range}(B_r)$   &\quad $H_{r,q}>0$  &\quad $H_{r,q}<0$\\
\hline\hline
$2<r$           &    $+1$  &  $\left(2,\frac23(r+1)\right)$  & $q\geq \frac23(r+1)$ &     $q\leq2$\\
\hline
$1<r<2$        &   $+1$  &  $\left(\frac23(r+1),2\right)$  & $q\geq 2$ &     $q\leq\frac23(r+1)$\\
\hline
$\frac12<r<1$&   $-1$    &  $\left(\frac23(r+1),2r\right)$ & $q\leq\frac23(r+1)$ &     $q\geq2r$\\
\hline
$0<r<\frac12$&   $-1$    &  $\left(2r,\frac23(r+1)\right)$ & $q\leq2r$ &     $q\geq \frac23(r+1)$\\
\hline
$-1<r<0$&   $+1$    &  $\left(0,\frac23(r+1)\right)$ & $q\geq \frac23(r+1)$ &     $q<0$\\
\hline
$r<-1$&   $+1$    &  $\left(\frac23(r+1),0\right) $ & $q>0$ &     $q\leq \frac23(r+1)$\\
\hline
\end{tabular}
\belowcaptionskip=0pt
\caption{The necessary and sufficient conditions on $r$ and $q$ that ensure the validity
of the inequalities  $H_{r,q}>0$ or $H_{r,q}>0$ on $(0,\infty)$.}
\label{tab2}
\end{center}
\end{table}
\noindent From the last two columns of Table~\ref{tab2}, the conclusion of Proposition~\ref{pro24} follows.
\end{proof}

\begin{corollary}\label{cor25}
Let $r$ and $q$ be two real numbers such that $r\notin\{0,1\}$ and $q\ne 0$, and 
let $F_{r,q}$ be the function defined on $J=(0,+\infty)$ by 
$$F_{r,q}(x)=\frac{\big(\cosh(r x)\big)^{q/r}-1}{\big(\cosh( x)\big)^{q}-1}.$$ 
Then, the following conclusion holds :
\begin{enumerate}[\upshape(a)]
\item  $F_{r,q}$ is decreasing if $(r,q)$ belongs to one of the following sets :

$\begin{matrix}
{\scriptscriptstyle \square}~&  \left\{(u,v):u<0,v\leq\min(0,\frac23(u+1))\right\}\setminus\{(-1,0)\},\hfill\\
{\scriptscriptstyle \square}~& \left\{(u,v):0<u<1,v\geq\max(2u,\frac23(u+1))\right\}\setminus\{(\frac12,1)\},\hfill\\
{\scriptscriptstyle \square}~& \left\{(u,v):1<u,v\leq\min(2,\frac23(u+1))\right\}\setminus\{(2,2)\}.\hfill
\end{matrix}$

\item  $F_{r,q}$ is increasing if $(r,q)$ belongs to one of the following sets :

$\begin{matrix}
{\scriptscriptstyle \square}~& \left\{(u,v):u<0,v\geq \max(0,\frac23(u+1))\right\}
\setminus\{(-1,0)\},\hfill\\
{\scriptscriptstyle \square}~& \left\{(u,v):0<u<1, v\leq\min(2u,\frac23(u+1))\right\}\setminus\{(\frac12,1)\},\hfill\\
{\scriptscriptstyle \square}~&\left\{(u,v):1<u,v\geq\max(2,\frac23(u+1))\right\}\setminus\{(2,2)\}.\hfill
\end{matrix}$
\end{enumerate}
\end{corollary}
\begin{proof} 
 Given nonzero $q$ and $r$, we define the function $f_{r,q}$ on $[0,+\infty)$ by
$$f_{r,q}(x)=(\cosh(rx))^{q/r}-1.$$
So that, for $x>0$, we have $F_{r,q}(x)=f_{r,q}(x)/f_{1,q}(x)$. Noticing that $f_{r,q}(0)=0$, and that, for $x>0$, we have
$$\frac{f^\prime_{r,q}(x)}{f^\prime_{1,q}(x)}=G_{r,q}(x)$$
where $G_{r,q}$ is the function defined in Proposition~\ref{pro24}, we conclude that the monotonicity of $F_{r,q}$ can be deduced from that of  $G_{r,q}$ using Lemma~\ref{lm21}. This proves Corollary~\ref{cor25}
\end{proof}

The next proposition, is the technical tool used in the proof of our Theorem~\ref{th33}.
\bigskip\goodbreak
\begin{proposition}\label{pro26}
 Let $q$ be a positive real number, and let $R_q(x,t)$ be the function defined on $(0,+\infty)\times(1,+\infty)$ by 
$$R_{q}(x,t)=(t-1)\coth((t-1)x)-(q+1)\tanh(x)+(q-t)\tanh (t x).$$ 
Let ${\mathcal A}_q$ be the subset of $(1,+\infty)$ defined by
$${\mathcal A}_q=\left\{t>1 : \forall\, x>0,\quad\frac{\partial R_q}{\partial t}(x,t)<0\right\}.$$
Then,
$${\mathcal A}_q=\left\{
\begin{matrix}
\left[1+\frac{3}{4}(q-2),+\infty\right)&\quad\hbox{ if }& \hfill q>2,\\
\,\\
\left(1,+\infty\right)&\quad \hbox{ if } &0<q\leq 2.
\end{matrix}\right.$$
\end{proposition}

\begin{proof}
It is convenient to write $\delta$ for $t-1$ so that $t=1+\delta$. With this notation, we have
\begin{align*}
\frac{\partial R_q}{\partial t}(x,t)&=\coth((t-1) x)-\frac{ (t-1) x}{\sinh^2((t-1)x)}
-\tanh (tx)+\frac{(q-t)x}{\cosh^2 (tx)}\\
&=\frac{\cosh x}{\sinh(\delta x)\cosh (tx)}
-\frac{\delta x}{\sinh^2(\delta x)}+\frac{(q-t)x}{\cosh^2 (t x)}\\
&=\frac{x}{\cosh^2 (tx)}\left( \frac{\cosh x\cosh (tx)}{x\sinh(\delta x)}
-\frac{\delta\cosh^2 (tx)}{\sinh^2(\delta x)}+q-t\right).
\end{align*}
Hence,
$$\frac{\partial R_q}{\partial t}(x,t)=\frac{x}{\cosh^2( tx)}(q-S(x,t)),$$
with
$$S(x,t)=t+\delta\frac{\cosh^2 (tx)}{\sinh^2(\delta x)}-\frac{\cosh x\cosh (tx)}{x\sinh(\delta x)}.$$
We come to the conclusion that
$$(t \in{\mathcal A}_q)\iff (\forall\,x>0,\quad q<S(x,t)).$$
\smallskip
Noting that $\lim_{x\to0}S(x,t)=\frac{2+4t}{3}$, we conclude that
\begin{equation*}\label{E:star}
t\in{\mathcal A}_q\Longrightarrow q\leq \frac{2+4t}{3}.\tag{1}
\end{equation*}

In what follows, we will prove that $S(x,t)>\frac{2+4t}{3}$ for all $x>0$. Given $t=1+\delta>1$ and $x>0$, let $W$ be defined by
$$W=6x \sinh^2((t-1)x)\left(S(x,t)-\frac{2+4t}{3}\right).$$
We can express $W$ as follows
\begin{align*}
W&=6x \sinh^2(\delta x)\left(\delta\frac{\cosh^2 (tx)}{\sinh^2(\delta x)}-
\frac{\cosh x\cosh (tx)}{x\sinh(\delta x)}-\frac{t+2}{3}\right)\\
&=6\delta x\cosh^2 (t x)-6\cosh x\cosh (t x)\sinh(\delta x)-2(t+2)x \sinh^2(\delta x)\\
&=x\big(4t-1+3\delta\cosh(2tx)-(t+2)\cosh(2\delta x)\big)-\frac{3}{2}\big(\sinh (2tx)+\sinh(2\delta x)-\sinh(2x)\big).
\end{align*}
It follows that
\begin{align*}
W&=x\left(4t-1+\sum_{n=0}^\infty(3\delta t^{2n}-(t+2)\delta^{2n})\frac{(2x)^{2n}}{(2n)!}\right)
-\frac{3}{2}\sum_{n=0}^\infty(t^{2n+1}+\delta^{2n+1}-1)
\frac{(2x)^{2n+1}}{(2n+1)!}\\
&=(4\delta+3)x+\sum_{n=0}^\infty\big((2n+1)(3\delta t^{2n}-(t+2)\delta^{2n})-3t^{2n+1}-3\delta^{2n+1}+3\big)\frac{2^{2n}x^{2n+1}}{(2n+1)!}
\end{align*}
Noting that the coefficients of $x$ and $x^3$ in this expansion are zero, we conclude that
$$W=\sum_{n=2}^\infty \frac{2^{2n}P_n(\delta)}{(2n+1)!}x^{2n+1},$$
where $P_n$ is the polynomial defined by
$$P_n(\delta)=(2n+1)\left(3\delta(1+\delta)^{2n}-(3+\delta)\delta^{2n}\right)-3(1+\delta)^{2n+1}
-3\delta^{2n+1}+3.$$
Clearly, $P_n(0)=0$. Moreover, we have

\begin{align*}
\frac{P_n^\prime(\delta)}{2n+1}&=
  3(1+\delta)^{2n}+6n\delta(1+\delta)^{2n-1}
  -\delta^{2n}-2n(3+\delta)\delta^{2n-1}-3(1+\delta)^{2n}-3\delta^{2n}\\
&= 2\delta\left(
  3n (1+\delta)^{2n-1}-n(3+\delta)\delta^{2n-2}-2\delta^{2n-1}
  \right)\\
&= 2\delta\left(
 3n (1+\delta)^{2n-1}-(n+2)\delta^{2n-1}-3n\delta^{2n-2}
\right)\\
&= 2\delta\big( 3n
\big((1+\delta)^{2n-1}-\delta^{2n-1}-(2n-1)\delta^{2n-2}\big)+2(n-1)\delta^{2n-1}+6n(n-1)\delta^{2n-2}
\big).
\end{align*}
Since $(1+\delta)^{m}-\delta^{m}-m\delta^{m-1}\geq0$ for $\delta\geq0$ and $m\geq1$, we conclude
that $P^\prime_n(\delta)>0$ for $\delta>0$, and consequently $P_n(\delta)>P_n(0)=0$ for $\delta>0$ and $n\geq2$. This proves that that $W>0$, and consequently,
$$\forall\,t>1,~\forall\, x>0,\qquad S(x,t)>\frac{2+4t}{3}.$$
Remembering \eqref{E:star} we see that for $t>1$ the condition $t\in {\mathcal A}_q$ is equivalent to $q\leq\frac{2+4t}{3}$, so we arrive to the following conclusion : 
$${\mathcal A}_q=\left\{
\begin{matrix}
\left[1+\frac{3}{4}(q-2),+\infty\right)&\quad\hbox{ if }& \hfill q>2,\\
\,\\
\left(1,+\infty\right)&\quad \hbox{ if } &0<q\leq 2.
\end{matrix}\right.$$
The proof of Proposition~\ref{pro26}  is complete.
\end{proof}

The following corollary is straightforward :

\begin{corollary}\label{cor27}
 Let $q$ be a positive real number, and let $R_q$ and ${\mathcal A}_q$ be  defined as in Proposition 2.6. Then, for every $k$ and $\ell$ in ${\mathcal A}_q$ such that $k<\ell$ we have
$$ \forall x>0,\qquad R_q(x,\ell)<R_q(x,k).$$
\end{corollary}

\section{The Main Theorems }\label{sec3}

In what follows the set of couples $(a,b)$ where $a$ and $b$ are positive real numbers such that $a\ne b$ will
be denoted by ${\mathcal D}$.
\smallskip

\begin{theorem}\label{th31}
Let $t$, $s$ and $p$ be nonzero real numbers with $s<t$.
For any $(a,b)$ from ${\mathcal D}$, we define the ratio of 
differences of Power Means $\rho(s,t,p;a,b)$ by
$$\rho(s,t,p;a,b)=\frac{M_s^p(a,b)-G^p(a,b)}{M_t^p(a,b)-G^p(a,b)},$$
\begin{enumerate}[\upshape (a)]
\item If $0<s<t$, then, for $p\geq\max(2s,\frac23(s+t))$, (with the exception that $p>2s$ if $t=2s$),
we have $$\forall (a,b)\in{\mathcal D},\qquad 2^{\frac{p}{t}-\frac{p}{s}}<\rho(s,t,p;a,b)<\frac{s}{t},$$
 and  for $p\leq\min(2s,\frac23(s+t))$, (with the exception that $p< 2s$ if $t=2s$), we have 
  $$\forall (a,b)\in{\mathcal D},\qquad\frac{s}{t} <\rho(s,t,p;a,b)<\min(1,2^{\frac{p}{t}-\frac{p}{s}}). $$
\item If $s<t<0$, then, for $p\geq\max(2t,\frac23(s+t))$, (with the exception that $p>2t$ if $s=2t$), we have $$\forall (a,b)\in{\mathcal D},\qquad\max(1,2^{\frac{p}{t}-\frac{p}{s}})<\rho(s,t,p;a,b)<\frac{s}{t},$$
 and  for $p\leq\min(2t,\frac23(s+t))$, (with the exception that $p<2t$ if $s=2t$),
we have $$\forall (a,b)\in{\mathcal D},\qquad\frac{s}{t}<\rho(s,t,p;a,b)<2^{\frac{p}{t}-\frac{p}{s}}.$$
\item  If $s<0<t$, then, for $p\geq\max(0,\frac23(s+t))$, (with the exception that $p>0$ if $s=-t$), we have 
                   $$\forall (a,b)\in{\mathcal D},\qquad\frac{s}{t}<\rho(s,t,p;a,b)<0,$$
        and  for $p\leq\min(0,\frac23(s+t))$, (with the exception that $p<0$ if $s=-t$), we have 
                   $$\forall (a,b)\in{\mathcal D},\qquad-\infty<\rho(s,t,p;a,b)<\frac{s}{t}.$$
\end{enumerate}
\noindent Moreover, all these inequalities are sharp.
\end{theorem}
\begin{proof}
 Let $F_{r,q}$ be the function defined in Corollary 2.5. If $F_{r,q}$ is strictly monotonous on the interval
 $(0,+\infty)$ then its bounds on this interval can be determined from its limits
at $0^+$  and at $+\infty$. In fact, it is straightforward to see that : 

$$\lim_{x\to0^+}F_{r,q}(x)=r,$$
and
$$\lim_{x\to\infty}F_{r,q}(x)=\left\{
\begin{matrix}
2^{q-q/r}&\hbox{if $r>0$ and $q>0$},\\
0&\hbox{if $r<0$ and $q>0$},\\
1&\hbox{if $r>0$ and $q<0$},\\
-\infty&\hbox{if $r<0$ and $q<0$}.
\end{matrix}
\right.$$
Now, consider two distinct positive real numbers $a$ and $b$. Since the ratio 
$\rho(s,t,p;a,b)$ is symmetric and homogeneous in $a$ and $b$,  we may suppose, without loss of generality, that $ab=1$ and $a>b$. Then, with
$$r=\frac{s}{t},\quad q=\frac{p}{t},\quad\hbox{ and }\quad x=\abs{t}\ln(a)$$
one checks easily that $\rho(s,t,p;a,b)=F_{r,q}(x)$.

\noindent(a)~ Let us suppose that $0<s<t$, then $r=\frac{s}{t}\in(0,1)$. Using Corollary 2.5 we come to the following
conclusion :\par
\begin{itemize}
\item  $F_{r,q}$ is decreasing on $(0,+\infty)$  if $q\geq\max(2r,\frac23(r+1))$ with $(r,q)\ne(\frac12,1)$. This proves
that
$$2^{\frac{p}{t}-\frac{p}{s}}<\rho(s,t,p;a,b)<\frac{s}{t}$$
if $p\geq\max(2s,\frac23(s+t))$ with the exception that $p>2s$ if $t=2s$.
\item  Also $F_{r,q}$ is increasing on $(0,+\infty)$ if $q\leq\min(2r,\frac23(r+1))$ with $(r,q)\ne(\frac12,1)$. This proves, distinguishing the cases $p<0$ and $p>0$, that
$$\frac{s}{t}<\rho(s,t,p;a,b)<\min(1,2^{\frac{p}{t}-\frac{p}{s}})$$
if $p\leq\min(2s,\frac23(s+t))$ with the exception that $p<2s$ if $t=2s$, This ends the proof of (a).
\end{itemize}
\noindent(b)~ Let us suppose that $s<t<0$, then $r=\frac{s}{t}\in(1,+\infty)$. Using Corollary 2.5 we come to the following conclusion :
\begin{itemize}
\item  $F_{r,q}$ is decreasing on $(0,+\infty)$ if $q\leq\min(2,\frac23(r+1))$ with $(r,q)\ne(2,2)$. This proves, distinguishing the cases $p<0$ and $p>0$, that
$$\max(1,2^{\frac{p}{t}-\frac{p}{s}})<\rho(s,t,p;a,b)<\frac{s}{t}$$
if $p\geq\max(2t,\frac23(s+t))$ with the exception that $p>2t$ if $s=2t$.
\item  Also $F_{r,q}$ is increasing on $(0,+\infty)$ if $q\geq\max(2,\frac23(r+1))$ with $(r,q)\ne(2,2)$. This proves that
$$\frac{s}{t}<\rho(s,t,p;a,b)<2^{\frac{p}{t}-\frac{p}{s}}$$
if $p\leq\min(2t,\frac23(s+t))$ with the exception that $p<2t$ if $s=2t$, This ends the proof of (b).
\end{itemize}
\noindent(c)~ Let us suppose that $s<0<t$, then $r=\frac{s}{t}\in(-\infty,0)$. Using Corollary 2.5 we come to the following conclusion :
\begin{itemize}
\item  $F_{r,q}$ is decreasing on $(0,+\infty)$ if $q\leq\min(0,\frac23(r+1))$ with $(r,q)\ne(-1,0)$. This proves
that
$$-\infty<\rho(s,t,p;a,b)<\frac{s}{t}$$
if $p\leq\min(0,\frac23(s+t))$ with the exception that $p<0$ if $s=-t$.
\item  Also $F_{r,q}$ is increasing on $(0,+\infty)$ if $q\geq\max(0,\frac23(r+1))$ with $(r,q)\ne(-1,0)$. This proves that
$$\frac{s}{t}<\rho(s,t,p;a,b)<0$$
if $p\geq\max(0,\frac23(s+t))$ with the exception that $p>0$ if $s=-t$, This ends the proof of (c).
\end{itemize}
\noindent The proof of Theorem~\ref{th31} is complete.
\end{proof}
\bigskip
The following Corollary corresponds to the case $p=1$.

\begin{corollary}\label{cor32}
Let $t$ and $s$ be nonzero real numbers with $s<t$.
\begin{enumerate}[\upshape (a)]
\item  If  $s<t<0$, then the following  holds
$$\forall (a,b)\in{\mathcal D},\qquad 1<\frac{M_s(a,b)-G(a,b)}{M_t(a,b)-G(a,b)}<\frac{s}{t}.$$
\item  If  $s<0<t\leq \frac32-s$, then the following holds
$$\forall (a,b)\in{\mathcal D},\qquad \frac{s}{t}<\frac{M_s(a,b)-G(a,b)}{M_t(a,b)-G(a,b)}<0.$$
\item  If one of the following conditions is satisfied
\begin{enumerate}[\upshape (i)]
\item $ 0<s<\frac12$, and $s<t\leq\frac32-s$,
\item $s=\frac12$,  and $\frac12<t<1$
\end{enumerate}
 then the following  holds
$$\forall (a,b)\in{\mathcal D},\qquad 2^{\frac1t-\frac1s}<\frac{M_s(a,b)-G(a,b)}{M_t(a,b)-G(a,b)}<\frac{s}{t}.$$
\item  If one of the following conditions is satisfied
\begin{enumerate}[\upshape (i)]
\item $ s=\frac12$, and $1<t$,
\item $\frac12<s<\frac34$, and $\frac32-s\leq t$,
\item $\frac34\leq s<t$,
\end{enumerate}
 then the following  holds
$$\forall (a,b)\in{\mathcal D},\qquad\frac{s}{t}<\frac{M_s(a,b)-G(a,b)}{M_t(a,b)-G(a,b)}<2^{\frac1t-\frac1s}.$$
\end{enumerate}
\noindent Moreover, all these inequalities are sharp.
\end{corollary}
\parindent=0pt

\noindent{\bf Remark.} In particular, choosing $s=1$ or $t=1$ in the above corollary we obtain that
$$\forall\,(a,b)\in{\mathcal D},\qquad 2^{1-1/s}<\frac{M_s(a,b)-G(a,b)}{A(a,b)-G(a,b)}<s$$
if $0<s<\frac12$ or $1<s$. Further, this inequality is reversed if $\frac12<s<1$, and we have
$$\forall\,(a,b)\in{\mathcal D},\qquad s<\frac{M_s(a,b)-G(a,b)}{A(a,b)-G(a,b)}<0$$
for $s<0$. This is the conclusion of Theorem 1. from \cite{wudeb}. 

 In The next theorem, we restrict our attention to positive parameters.

\begin{theorem}\label{th33}
Let $r$, $t$, $s$ and $p$ be four real numbers satisfying the following two conditions:\par
\begin{enumerate}[\upshape (i)]
\item   $0<s<t<r$,
\item   $0<p\leq \frac{4t+2s}{3}.$.
\end{enumerate}
\noindent Then, for any positive numbers $a$ and $b$ such that $a\ne b$ we have
$$\frac{2^{-p/r}-2^{-p/s}}{2^{-p/t}-2^{-p/s}}<\frac{M_r^p(a,b)-M_s^p(a,b)}{M_t^p(a,b)-M_s^p(a,b)}<\frac{r-s}{t-s}.$$
\end{theorem}

\begin{proof} 
Let us consider, for $q>0$ the function $F_q$ defined on $[0,+\infty)\times (1,+\infty)$ by
$$F_q(x,v)=\left(\frac{(\cosh(v x))^{1/v}}{\cosh x}\right)^q.$$

Note that, for a given $v>1$, we have 
\begin{equation*}\label{E:dag}
\frac{\partial F_q}{\partial x}(x,v)=(F_q)_x^\prime(x,v)=\frac{q\sinh((v-1)x)}{\cosh x\,\cosh (v x)}F_q(x,v).\tag{2}
\end{equation*}
This proves that $x\mapsto F_q(x,v)$ is increasing on $[0,+\infty)$, it defines a bijection from
$[0,+\infty)$ onto $[1,2^{q-q/v})$. 

From \eqref{E:dag} we conclude that, for $x>0$ we have
$$\frac{(F_q)^{\prime\prime}_{xx}(x,v)}{(F_q)^\prime_x(x,v)}
= (v-1)\coth((v-1)x)-(q+1)\tanh(x)+(q-v)\tanh (vx).$$
That is
$$\frac{(F_q)^{\prime\prime}_{xx}(x,v)}{(F_q)^\prime_x(x,v)}=R_q(x,v)$$
where $R_q$ is the function defined in Proposition~\ref{pro26}.

Now, consider two elements $k$ and $\ell$  from ${\mathcal A}_q$ such that $k<\ell$. Using Corollary~\ref{cor27}, we conclude that
$$\forall\,x>0,\quad
\frac{(F_q)^{\prime\prime}_{xx}(x,\ell)}{(F_q)^\prime_x(x,\ell)}<\frac{(F_q)^{\prime\prime}_{xx}(x,k)}{(F_q)^\prime_x(x,k)},
$$
or,
$$\forall\,x>0,\quad
\frac{\partial}{\partial x}\left(\frac{(F_q)^{\prime}_x(x,\ell)}{(F_q)^{\prime}_x(x,k)}\right)<0.$$
It follows that 
$x\mapsto (F_q)^{\prime}_x(x,\ell)/(F_q)^{\prime}_x(x,k)$ is decreasing on $(0,+\infty)$, and by Lemma~\ref{lm21}, we arrive to the conclusion that
$x\mapsto \frac{F_q(x,\ell)-F_q(0,\ell)}{F_q(x,k)-F_q(0,k)}$ is  decreasing on $(0,+\infty)$.

Noting that
$$\lim_{x\to0^+}\frac{F_q(x,\ell)-1}{F_q(x,k)-1}=\frac{\ell-1}{k-1}
\quad\hbox{and}\quad
\lim_{x\to+\infty}\frac{F_q(x,\ell)-1}{F_q(x,k)-1}=\frac{2^{q-q/\ell}-1}{2^{q-q/k}-1}$$
 we conclude that, for every $x>0$ we have
\begin{equation*}\label{E:ddag}
\frac{2^{q-q/\ell}-1}{2^{q-q/k}-1}<\frac{F_q(x,\ell)-1}{F_q(x,k)-1}<\frac{\ell-1}{k-1}.\tag{3}
\end{equation*}
which is valid for every $k$ and $\ell$  from ${\mathcal A}_q$ such that $k<\ell$.

Consider, now, two distinct positive real numbers $a$ and $b$. 
Since the ratio 
$$\frac{M_r^p(a,b)-M_s^p(a,b)}{M_t^p(a,b)-M_s^p(a,b)}$$ is symmetric in $a$ and $b$ and homogeneous, without loss of generality, we may suppose that $ab=1$ and $a>b$. Then, we define
$$\ell=\frac{r}{s},\quad k=\frac{t}{s},\quad q=\frac{p}{s},\quad\hbox{ and }\quad x=\ln(a^s)$$
and we apply \eqref{E:ddag} to this data to obtain the desired conclusion.
\end{proof}
 
\noindent{\bf Remarks.} 

\begin{itemize}
\item   Theorem 2 in \cite{wudeb} gives the same inequality, with our condition (ii) replaced by a stronger condition, namely $p\in(0,t]$. Therefore, our Theorem~\ref{th33} is a refinement upon that theorem.
\item  Setting $t=1$ and letting $s$ tend to $0^+$ in Theorem~\ref{th33}, we obtain 
$$\forall(a,b)\in{\mathcal D},\qquad 2^{p-p/r} \leq\frac{M_r^p(a,b)-G^p(a,b)}{A^p(a,b)-G^p(a,b)}\leq r.$$
for $0<p\leq 4/3$, $r>1$.
In fact, this follows also -with strict inequalities- from Theorem~\ref{th31}(a).
 
\item  Again, Setting $p=1$ and letting $s$ tend to $0^+$ in Theorem~\ref{th33},  we obtain 
$$\forall(a,b)\in{\mathcal D},\qquad 2^{1/t-1/r} \leq\frac{M_r(a,b)-G(a,b)}{M_t(a,b)-G(a,b)}\leq \frac{r}{t}.$$
for $r>t\geq 3/4$. This follows also -with strict inequalities-  from Corollary~\ref{cor32}(d).
\end{itemize}

\end{document}